\documentclass[12p]{article}
\usepackage[latin1]{inputenc}
\usepackage{a4wide}
\usepackage{amsmath}
\usepackage{amssymb}
\usepackage{amsfonts,euler}
\usepackage{amsthm}
\usepackage{url}
\usepackage{graphicx}
\usepackage{nicefrac}

\newcommand{\E}{\ensuremath{\mathbb{E}}}
\newcommand{\V}{\ensuremath{\mathrm{Var}}}
\newcommand{\Prob}{\ensuremath{\mathbb{P}}}

\begin{document}
\newtheorem{thm}{Theorem}[section]
\newtheorem{lem}{Lemma}[section]
\newtheorem{cor}[thm]{Corollary}
\newtheorem{rem}{Remark}
\newtheorem{proposition}[thm]{Proposition}

\newcommand{\Rset}{\ensuremath{\mathbb{R}}}
\newcommand{\Nset}{\ensuremath{\mathbb{N}}}
\newcommand{\bo}{\ensuremath{\mathrm{O}}}

\title{\bf Refined Quicksort asymptotics}
\author{Ralph Neininger\\
Institute for Mathematics\\
J.W.~Goethe University Frankfurt\\
60054 Frankfurt am Main\\
Germany\\[2ex]
Email: neiningr@math.uni-frankfurt.de}

\date{January 24, 2013}
\maketitle

\begin{abstract}
The  complexity of the Quicksort algorithm is usually measured by the number of key comparisons used during its execution. When operating on a list of $n$ data, permuted uniformly at random, the appropriately normalized complexity $Y_n$ is known to converge almost surely to a non-degenerate random limit $Y$.  This assumes a natural embedding of all $Y_n$ on  one probability space, e.g.,  via  random binary search trees. In this note a central limit theorem for the error term in the latter almost sure convergence is shown:
$$\sqrt{\frac{n}{2\log n}}(Y_n-Y) \stackrel{d}{\longrightarrow}  {\cal N} \qquad (n\to\infty),$$ 
where ${\cal N}$ denotes a standard normal random variable.
 \end{abstract}

\noindent
{\em  AMS 2010 subject classifications.} 60F05,  60F15, 68P10, 68Q25.  \\
{\em Key words.} Quicksort, complexity, key comparisons, central limit theorem, strong limit theorem, rate of convergence, Zolotarev metric, contraction method.

\section{Introduction and result}
Quicksort, invented by Hoare \cite{Ho62}, is one of the most
widely used algorithms for sorting. Given a list $\Gamma=(u_1,\ldots,u_n)\in\Rset^n$, Quicksort starts picking a key (i.e., an element), say the first one $u_1$, as ``pivot'' element.  The other keys in $\Gamma$  are then partitioned into  lists $\Gamma_\le$ and $\Gamma_>$. Key $u_j$ is contained in list $\Gamma_\le$ if the ``key comparison'' between the pivot element $u_1$ and $u_j$ yields
 $u_j\le u_1$, otherwise $u_j$ is contained in list $\Gamma_>$, $2\le j\le n$. Finally, the lists $\Gamma_\le$ and $\Gamma_>$  are each sorted recursively unless their size is $0$ or $1$. 

The complexity of Quicksort is most commonly  measured by the total number of key comparisons used, although other cost measures have been studied as well. To capture the typical complexity of the algorithm it is usually assumed that the ranks of the elements in $\Gamma$ form a random, uniformly distributed permutation of $\{1,\ldots,n\}$. Subsequently this model assumption is met by starting with the list  $\Gamma=(U_1,\ldots,U_n)$, where $(U_j)_{j\ge 1}$ is a sequence of independent random variables,  identically distributed  with the uniform distribution on $[0,1]$. To be definite about the partitioning phase of the algorithm we assume that the order of elements in $\Gamma$ is preserved within the lists $\Gamma_\le$ and $\Gamma_>$, e.g., list $\Gamma=(4,2,5,6,1,8,3,7)$ is partitioned into the lists $\Gamma_\le=(2,1,3)$ and $\Gamma_>=(5,6,8,7)$. This property is shared by standard implementations when always using the first element as pivot element, see as general reference Mahmoud \cite{Ma00}.

We denote by $K_n$ the number of key comparisons used by Quicksort to sort the list 
$\Gamma=(U_1,\ldots,U_n)$, $n\ge 1$, and set $K_0:=0$. In the probabilistic analysis of the complexity of Quicksort often  characteristics of $K_n$ are studied that only depend  on the distribution ${\cal L}(K_n)$ of $K_n$. With respect to weak convergence  such results are reviewed below. 

However, in the present setting the $K_n$ are constructed on a joint probability space which in fact is a formulation via random binary search trees  discussed in Section \ref{asc} and used in the subsequent analysis. Hence, for $(K_n)_{n\ge 0}$ also path properties (in particular strong limit theorems) 
can be studied.  R\'{e}gnier \cite{re89} showed that   
\begin{align}\label{def_yn}
Y_n:= \frac{K_n -\E[K_n]}{n+1}, \quad n\ge 0,
\end{align}
is a martingale, which converges towards a random, non-degenerate limit $Y$ almost surely and in $L_p$:
\begin{align*}
\|Y_n-Y\|_p \to 0  \quad (n\to \infty),
\end{align*}
for any $1\le p<\infty$, where we denote $\|X\|_p:=\E[|X|^p]^{1/p}$ for a random variable $X$. 
(The case $p=2$ is explicitly discussed in \cite{re89}.) The mean of $K_n$ is  $\E[K_n]=2(n+1)H_n -4n$, where $H_n:=\sum_{k=1}^n 1/k$ denotes the $n$th harmonic number. Another proof  for the almost sure convergence of $(Y_n)_{n\ge 0}$ via a Doob-Martin compactification is given in Gr\"ubel \cite{Gr12}, see also Evans, Gr\"ubel and Wakolbinger \cite{EvGrWa12}.

R\"osler \cite{ro91} gave a  proof based on a contraction argument  for the convergence in distribution of $Y_n$ towards $Y$ and found that the limit distribution ${\cal L}(Y)$ satisfies
\begin{align} \label{def_c_func_a}
{\cal L}(Y) = {\cal L}(UY'+(1-U)Y''+C(U)), 
\end{align}
with, for $x\in[0,1]$,
\begin{align} \label{def_c_func_b}
C(x):=1+2x\log(x)+2(1-x)\log(1-x),
\end{align}
where $U,Y'$ and $Y''$ are independent, $Y'$ and $Y''$ are distributed as $Y$ and $U$ is uniformly distributed on $[0,1]$. 

The rate of the convergence $Y_n\to Y$ has been bounded, regarding the distributions ${\cal L}(Y_n)$ and ${\cal L}(Y)$, by various distance measures. 
The  minimal $L_p$-metric $\ell_p$ is given by 
\begin{align}\label{ell_p}
\ell_p(V,W):= \ell_p({\cal L}(V),{\cal L}(W)):= \inf\{\|V'-W'\|_p: {\cal L}(V)={\cal L}(V'), {\cal L}(W)={\cal L}(W')\},
\end{align}
for all $1\le p<\infty$ and random variables $V$, $W$ with $\|V\|_p, \|W\|_p<\infty$.  Note that the infimum in (\ref{ell_p})
 is over all joint distributions ${\cal L}(V',W')$ with the given marginals ${\cal L}(V)$ and ${\cal L}(W)$. 
 Fill and Janson \cite{FJ02} obtained for all $2\le p<\infty$ the  bounds
\begin{align*}
\ell_p(Y_n,Y)=\bo\left(\frac{1}{\sqrt{n}}\right),\quad  \ell_p(Y_n,Y) = \Omega\left(\frac{\log n}{n}\right),    
\end{align*}
as well as the explicit bound $\ell_2(Y_n,Y)< 2/\sqrt{n}$ for all $n\ge 1$. 

We denote by $F_V$ the distribution function of a random variable $V$. Then, for the Kolmogorov--Smirnov distance 
(uniform distance)
\begin{align*}
\varrho(V,W):=\varrho({\cal L}(V),{\cal L}(W)):= \sup_{x\in\Rset} |F_V(x)-F_W(x)|
\end{align*}
Fill and Janson \cite{FJ02} obtained for all $\varepsilon>0$ that
\begin{align*}
\varrho(Y_n,Y)=\bo\left(n^{\varepsilon-(1/2)}\right),\quad \varrho(Y_n,Y)=\Omega\left(\frac{1}{n}\right)
\end{align*}
together with the explicit lower bound $\varrho(Y_n,Y)\ge 1/(8(n+1))$ for all $n\ge 1$.

For the Zolotarev metric $\zeta_3$ defined in Section \ref{seczo}  Neininger and R\"uschendorf \cite{NR02} obtained the order 
\begin{align}\label{rate_zolo}
\zeta_3(Y_n,Y)=\Theta\left(\frac{\log n}{n}\right).
\end{align}
The techniques of \cite{NR02} are sufficiently sharp to obtain $\zeta_{2+\alpha}(Y_n,Y)=\Theta((\log n)/n)$ for all $\alpha \in (0,1]$ as well. Using inequalities between probability metrics,  based upon  (\ref{rate_zolo}), a couple of  upper bounds  for related distance measures between ${\cal L}(Y_n)$ and ${\cal L}(Y)$ were obtained in Section 3 of  \cite{NR02}.

The results mentioned above bound distances between the distributions of the $Y_n$ 
and $Y$. However, the  embedding of the $K_n$ on one probability space allows to measure the 
approximation given by the martingale convergence $Y_n \to Y$ as well. Very recently, Bindjeme and Fill \cite{BiFi12} started quantifying the almost sure convergence  $Y_n \to Y$ by identifying the $L_2$-distance between $Y_n$ and $Y$ exactly and asymptotically:
\begin{align}\label{var_asy_exp}
\|Y_n-Y\|_2=\left(\frac{1}{n+1}\left( 2H_n + 1 + \frac{6}{n+1} \right) -4\sum_{k=n+1}^\infty \frac{1}{k^2}\right)^\frac{1}{2} \sim \sqrt{\frac{2\log n}{n}}.
\end{align}

In the present note the error term $Y_n-Y$ is further studied with respect to its asymptotic distribution: 
\begin{thm} \label{main_thm}
Let the data $(U_i)_{i\ge 1}$ be a sequence of independent and identically distributed random variables each with the uniform distribution  on $[0,1]$.
For the number $K_n$ of key comparisons needed by the Quicksort algorithm to sort the list 
$(U_1,\ldots,U_n)$ and the almost sure limit $Y$ of $Y_n$ defined in {\rm (\ref{def_yn})}  we have, as $n\to\infty$, that 
\begin{align*}
\sqrt{\frac{n}{2\log n}}\left(Y_n - Y\right) \stackrel{d}{\longrightarrow} {\cal N}.
\end{align*}
\end{thm} 
The methods used for the proof of Theorem \ref{main_thm} also imply convergence of the third absolute moments, which yields an asymptotic expression for the $L_3$-distance between $Y_n$ and $Y$: 
\begin{cor}\label{coro}
For the normalized number $Y_n$ of key comparisons needed by the Quicksort algorithm and its almost sure limit $Y$  as in Theorem \ref{main_thm} we have, as $n\to\infty$, that
\begin{align*}
\left\|Y_n - Y\right\|_3 \sim  \frac{2}{\pi^{1/6}} \sqrt{\frac{\log n}{ n}}. 
\end{align*}
\end{cor}

\noindent
{\bf Notation.} Throughout, by $\stackrel{d}{\longrightarrow}$ convergence in distribution is denoted,  by ${\cal N}$ a random variable with the standard normal distribution. The Bachmann--Landau symbols are used in asymptotic statements. We denote by $\Nset$ the positive integers and let $\Nset_0:=\Nset \cup \{0\}$. By $\mathrm{B}(n,u)$ the binomial distribution with $n\in\Nset$ trials and success probability $u\in[0,1]$ is denoted, and by $\log x$ the natural logarithm of $x$ for $x>0$. Further $x\log x:=0$ for $x=0$ is used.

\subsubsection*{Acknowledgment} I thank Henning Sulzbach and Kevin Leckey for comments on a draft of this note and three anonymous referees for their careful reading and constructive remarks.

\section{Proof} \label{sec2}
The outline of the proof is as follows:
First, in Section \ref{asc} an explicit construction of $Y_n$ and $Y$ is recalled, which leads to a sample-pointwise recurrence relation for $Y_n-Y$. Then ideas from the contraction method are
 used for this recurrence. Compared to standard applications of the contraction method, mainly  additional dependencies between the arising random variables need  to be controlled, see the discussion at the end of Section \ref{asc}. This is done by use of inequalities for the Zolotarev metric, provided in Section \ref{seczo}. After technical preparations in Section \ref{seclem} the convergence in  Theorem \ref{main_thm} then is shown in Section \ref{pfpf} within the Zolotarev metric, which implies the stated convergence in distribution.
\subsection{Almost sure construction} \label{asc}
An explicit construction for the limit distribution ${\cal L}(Y)$  was given by R\"osler \cite{ro91} and recently linked  to the martingale limit $Y$ by Bindjeme and Fill \cite{BiFi12}. Since below, as a starting point, the same recursive equation (\ref{rec_bifi}) for $Y_n-Y$ is used as in  \cite{BiFi12}, for the reader's convenience, some notation is adopted from there. 

Consider the rooted complete infinite binary tree, where the root is labeled by the empty word $\epsilon$ and left and right children of a node labeled $\vartheta$ are labeled by the extended words $\vartheta 0$ and $\vartheta 1$ respectively. The set of labels is denoted by $\Theta:=\cup_{k=0}^\infty \{0,1\}^k$. The length $|\vartheta|$ of a label of a node is identical to the depth of the node in the rooted complete infinite binary tree. Now the sequence of keys $(U_j)_{j\ge 1}$ is inserted into the rooted  infinite binary tree according to the binary search tree algorithm: The first key $U_1$ is inserted in the root and occupies the root. Then we successively insert the following keys, where each key traverses the occupied nodes starting at the root. Whenever the key traversing is less than the occupying key at a node it moves on to the left child of that node, otherwise to its right child. The first empty node visited is occupied by the key. General references for search tree algorithms are Knuth \cite{Kn98} or Mahmoud \cite{Ma92}.
We denote by $V_\vartheta$ the key which occupies the node labeled $\vartheta$. In particular we have $V_\epsilon=U_1$. 

Further, we associate with each node labeled $\vartheta$ an interval $I_\vartheta$ defined recursively: We set $I_\epsilon:=[0,1]$. Assume that $I_\vartheta=[L_\vartheta,R_\vartheta]$ is already defined for some $\vartheta\in \Theta$. Then we set $I_{\vartheta 0}:=[L_\vartheta,V_\vartheta]$ and $I_{\vartheta 1}:=[V_\vartheta,R_\vartheta]$. Note that  by construction we always have $V_\vartheta \in I_\vartheta$.  The relative positions of $V_\vartheta$ within $I_\vartheta$ are crucial: We denote the interval lengths by $\varphi_\vartheta:= R_\vartheta-L_\vartheta$ and
\begin{align*}
\Upsilon_\vartheta:= \frac{V_\vartheta-L_\vartheta}{R_\vartheta-L_\vartheta}=\frac{\varphi_{\vartheta 0}}{\varphi_\vartheta}, \quad \vartheta\in\Theta.
\end{align*}
By construction, $(\Upsilon_\vartheta)_{\vartheta\in\Theta}$ is a family of independent random variables, identically distributed with the uniform distribution on $[0,1]$. 
Furthermore, with the function $C$ defined in (\ref{def_c_func_b}) we set 
\begin{align*}
G_\vartheta:= \varphi_\vartheta C(\Upsilon_\vartheta).
\end{align*}
In a related setting R\"osler \cite{ro91} showed that the series
\begin{align}\label{limit_series}
\sum_{j=0}^\infty \sum_{\vartheta\in \Theta: |\vartheta|=j} G_\vartheta
\end{align}
is convergent in any $L_p$ with $p<\infty$ and that it
has the same distribution as the martingale limit $Y$. 

Bindjeme and Fill \cite{BiFi12} showed that the random variable in (\ref{limit_series}) is even almost surely identical to $Y$. 
Moreover, they use the latter construction to give the following sample-pointwise extension of the distributional identity (\ref{def_c_func_a}) for $Y$: Roughly, the left and right subtree of the root, i.e., the complete infinite binary trees rooted at the nodes labeled $0$ and $1$ get all their nodes' interval lengths renormalized by $1/U_1$ and $1/(1-U_1)$ respectively. This unwinds the original dependence between interval lengths from nodes of the left and right subtree induced from $U_1$ and allows an almost sure construction of the distributional identity (\ref{def_c_func_a}). Formally,  
with the root key $V_\epsilon=U_1$, we define for all $\vartheta\in\Theta$ random variables 
\begin{align*}
\varphi^{(0)}_\vartheta := \frac{1}{U_1}\varphi_{0\vartheta}=\frac{\varphi_{0\vartheta}}{\varphi_{0}}   ,\quad 
\varphi^{(1)}_\vartheta := \frac{1}{1-U_1}\varphi_{1\vartheta}=\frac{\varphi_{1\vartheta}}{\varphi_{1}},
\end{align*}
and
\begin{align*}
G^{(i)}_\vartheta :=\varphi^{(i)}_\vartheta C(\Upsilon_{i\vartheta}), \quad 
Y^{(i)}:= \sum_{j=0}^\infty \sum_{\vartheta\in \Theta: |\vartheta|=j} G^{(i)}_\vartheta,\quad i\in\{0,1\}.
\end{align*}
Then, cf.~Proposition 2.1 in \cite{BiFi12}, we have 
\begin{align}\label{rec_y}
Y= U_1 Y^{(0)} + (1-U_1)Y^{(1)} + C(U_1),
\end{align}
and $U_1$, $Y^{(0)}$, $Y^{(1)}$ are independent and $Y^{(0)}$ and  $Y^{(1)}$ have the same distribution as $Y$.

Now, we denote by $I_n$ the number of keys among $U_1,\ldots,U_n$ that are inserted in the left subtree of the root, i.e., the subtree rooted at the node labeled $0$. Note that $I_n$ takes values in $\{0,\ldots,n-1\}$ and, conditional on $U_1=u$, we have for $I_n$ the binomial $\mathrm{B}(n-1,u)$ distribution. Furthermore, denote by $K_{n,0}$ and $K_{n,1}$ the number of key comparisons used to sort the left and right sublists $\Gamma_\le$ and $\Gamma_>$ generated when first partitioning $(U_1,\ldots,U_n)$. 
Note that the sizes of  $\Gamma_\le$ and $\Gamma_>$ are $I_n$ and $n-1-I_n$, respectively. Since the first partitioning phase of Quicksort requires $n-1$ key comparisons we have for all $n\ge 1$ that
\begin{align}\label{def_xn}
K_n = K_{n,0} + K_{n,1} + n-1.
\end{align}
Recall the normalization (\ref{def_yn}) for $K_n$. Hence, with $\mu(n):=\E[K_n]$ we define  normalizations of $K_{n,0}$ and $K_{n,1}$ by
\begin{align}\label{def_scal}
Y_{n,0} := \frac{K_{n,0} - \mu(I_n)}{I_n+1}, \quad  Y_{n,1} := \frac{K_{n,1} - \mu(n-1-I_n)}{n-I_n}.
\end{align}
(To be clear about the notation, we have $\mu(I_n)=\E[K_{I_n}\,|\, I_n]$ and in general  $\mu(I_n)\neq \E[K_{I_n}]$.)
Note that conditional on $I_n=j$ we have that $Y_{n,0}$ and $Y_{n,1}$ are independent and have the same distributions as $Y_j$ and $Y_{n-1-j}$, respectively.
 From (\ref{def_yn}), (\ref{def_xn}) and (\ref{def_scal})
we obtain the (sample-pointwise) recurrence, cf.~equation (2.4) in \cite{BiFi12},
\begin{align}\label{rec_yn}
Y_n= \frac{I_n+1}{n+1}Y_{n,0}+ \frac{n-I_n}{n+1} Y_{n,1} + \frac{n}{n+1}C_n(I_n+1), \quad n\ge 1,
\end{align}
where, for $1\le i\le n$ we define
\begin{align*}
C_n(i):=\frac{1}{n}\left(\mu(i-1)+\mu(n-i)-\mu(n)+n-1\right).
\end{align*}
Altogether, (\ref{rec_y}) and (\ref{rec_yn}) yield a recurrence for the error term under consideration in Theorem \ref{main_thm}, for all $n\ge 1$, cf.~equation (2.6) in \cite{BiFi12}, 
\begin{align}
Y_n -Y &= \frac{I_n+1}{n+1}\left(Y_{n,0}-Y^{(0)}\right) + 
\frac{n-I_n}{n+1}\left(Y_{n,1}-Y^{(1)}\right) + \left(\frac{I_n+1}{n+1} - U_1\right) Y^{(0)} \nonumber \\
&\;\;\;~
+ \left(\frac{n-I_n}{n+1} -(1-U_1)\right) Y^{(1)}+\frac{n}{n+1}C_n(I_n+1)-C(U_1). \label{rec_bifi}
\end{align}

Note that $Y_n -Y$ is already centered and has variance, see (\ref{var_asy_exp})
\begin{align} \label{var_asy}
\sigma^2(n):=\V(Y_n -Y) = \|Y_n-Y\|_2^2 \sim \frac{2\log n}{n} \quad (n\to\infty),
\end{align}
and $\sigma(n)>0$ for all $n\ge 0$.
Hence, with the scaling
\begin{align} \label{scaling}
X_n:=\frac{Y_n -Y}{\sigma(n)}, \quad n\ge 0,
\end{align}
we obtain for all $n\ge 1$ that 
\begin{align} \label{basic_eqn}
X_n= A_0^{(n)} \frac{1}{\sigma(I_n)}\left(Y_{n,0}-Y^{(0)}\right) 
+ A_1^{(n)}\frac{1}{\sigma(n-1-I_n)}\left(Y_{n,1}-Y^{(1)}\right) + b^{(n)},
\end{align}
where
\begin{align}
A_0^{(n)}&:= \frac{(I_n+1)\sigma(I_n)}{(n+1)\sigma(n)}, \qquad    
A_1^{(n)}:= \frac{(n-I_n)\sigma(n-1-I_n)}{(n+1)\sigma(n)}, \nonumber\\
b^{(n)}&:= \frac{1}{\sigma(n)}\left[\left(\frac{I_n+1}{n+1} - U_1\right) Y^{(0)}
+ \left(\frac{n-I_n}{n+1} -(1-U_1)\right) Y^{(1)} \nonumber \right.\\
&\left. \phantom{\frac{1}{\sigma(n)}}\quad\quad ~ +\frac{n}{n+1}C_n(I_n+1) -C(U_1)   \right].
\label{def_bn}
\end{align}
The asymptotics of $\sigma(n)$ in (\ref{var_asy}) and the fact that $I_n$, conditionally on $U_1=u$, has the binomial $\mathrm{B}(n-1,u)$ distribution imply together with the strong law of large numbers and dominated convergence that,
 as $n\to\infty$,
\begin{align} \label{asy_norm}
\left\|A_0^{(n)}-\sqrt{U_1}\right\|_p \to 0, \quad   \left\|A_1^{(n)}-\sqrt{1-U_1}\right\|_p \to 0 
\end{align}
for all $1\le p<\infty$.\\

\noindent
{\bf Remark.} Note that  convergence theorems from the contraction method, e.g., Corollary 5.2 in  \cite{NeRu04}, do not in general apply to recurrence (\ref{basic_eqn}). The reason is that  each of the random variables 
\begin{align*}
\frac{1}{\sigma(I_n)}\left(Y_{n,0}-Y^{(0)}\right), \quad  
\frac{1}{\sigma(n-1-I_n)}\left(Y_{n,1}-Y^{(1)}\right)
\end{align*}
is conditionally on $I_n$ still (stochastically) dependent on $b^{(n)}$,  via the joint occurrence 
of $Y^{(0)}$ and $Y^{(1)}$ respectively, while in typical theorems from the contraction method  conditional independence is assumed.

\subsection{The Zolotarev metric}\label{seczo}
The proof of Theorem \ref{main_thm} in Section \ref{pfpf} is based on showing appropriate convergence within the Zolotarev metric and using that convergence in the Zolotarev metric implies weak convergence. The Zolotarev metric  has been 
studied in the context of distributional recurrences systematically in \cite{NeRu04}. We  collect the properties that are used subsequently, which can be found in Zolotarev \cite{zo76,zo77} if not stated otherwise.
For distributions 
${\cal L}(V)$, ${\cal L}(W)$ on $\Rset$ the Zolotarev distance $\zeta_s$, $s>0$, is defined by

\begin{equation}
\label{eq:3.6_app} \zeta_s(V,W) := \zeta_s({\cal L}(V),{\cal L}(W)):=\sup_{f\in {\cal F}_s}|\E[f(V) -
f(W)]|
\end{equation}
where $s=m+\alpha$ with $0<\alpha\le 1$ and
$m\in\Nset_0$. Here
\begin{equation} 
{\cal F}_s:=\{f\in
C^m(\Rset,\Rset):|f^{(m)}(x)-f^{(m)}(y)|\le
|x-y|^\alpha\}
\end{equation}
denotes  the space of $m$-times
continuously differentiable functions from
$\Rset$ to $\Rset$ such that the $m$-th
derivative is H\"older continuous of order
$\alpha$ with H\"older-constant $1$. 
We have that $\zeta_s(V,W)<\infty$ if (i) all 
moments of orders $1,\ldots,m$ of $V$ and $W$ are
equal and (ii) the $s$-th absolute moments of $V$ and
$W$ are finite.  Since later on only the case $s=3$ is
used, for finiteness of  $\zeta_3(V,W)$ it is thus sufficient 
 that  mean and variance of 
$V$ and $W$ coincide and both have a finite absolute moment of
 order $3$. A pair $(V,W)$ satisfying these moment assumptions subsequently is called
 {\em $\zeta_3$-compatible}, a term  not in use elsewhere. In particular, for fixed $\mu \in \Rset$ and $\sigma>0$, within the space of distributions
\begin{align*}
{\cal M}_3(\mu,\sigma^2):=\{ {\cal L}(V)\,:\, \E[V]=\mu, \V(V)=\sigma^2, \E[|V|^3]<\infty\}
\end{align*}
all pairs are $\zeta_3$-compatible and $({\cal M}_3(\mu,\sigma^2), \zeta_3)$ is a complete metric space. For the completeness (not used subsequently) see  \cite[Theorem 5.1]{DrJaNe08}.
 Convergence
in $\zeta_3$ implies weak convergence on $\Rset$.

Furthermore,  $\zeta_3$ is $(3,+)$ ideal, i.e., 
\begin{eqnarray*}
\zeta_3(V+Z,W+Z)\le\zeta_3(V,W), \quad  \zeta_3(cV,cW) = c^3 \zeta_3(V,W)
\end{eqnarray*}
for all  $Z$ being independent of $(V,W)$ and all $c>0$.
This in particular implies that for independent pairs $(V_1,V_2)$, $(W_1,W_2)$ such that both pairs are $\zeta_3$-compatible we have 
\begin{align}\label{sum_bound}
\zeta_3(V_1+W_1,V_2+W_2)&\le \zeta_3(V_1+W_1,V_2+W_1) + \zeta_3(V_2+W_1,V_2+W_2) \nonumber \\
&\le \zeta_3(V_1,V_2) + \zeta_3(W_1,W_2).
\end{align}
The metric $\zeta_3$ can be upper-bounded in terms of  the minimal $L_3$-metric $\ell_3$ defined in (\ref{ell_p}):
For $\zeta_3$-compatible $(V,W)$ we have, see  \cite[Lemma 2.1]{NR02},
\begin{align}\label{est_zeta_ell}
\zeta_3(V,W)\le \frac{1}{2}\left(\|V\|_3^2+ \|V\|_3\|W\|_3 + \|W\|_3^2\right)\ell_3(V,W).
\end{align}

Finally, a substitute for  (\ref{sum_bound}) when the independence  assumption there is violated is later used: 
\begin{lem} \label{zolo_est}
Let $V_1,V_2,W_1,W_2$ be random variables such that $(V_1,V_2)$ is $\zeta_3$-compatible and $(V_1+W_1,V_2+W_2)$  is $\zeta_3$-compatible. Then we have
\begin{align}
\zeta_3(V_1+W_1, V_2+W_2)  \le \zeta_3(V_1, V_2)+ \sum_{i=1}^2 \left\{ \frac{\|V_i\|_3^2\|W_i\|_3}{2}+   \frac{\|V_i\|_3\|W_i\|_3^2}{2} +  \frac{\|W_i\|_3^3}{6} \right\}.
\end{align}
\end{lem}
\begin{proof}
By the assumptions on $\zeta_3$-compatibility we have that the $\zeta_3$-distances appearing in the formulation of the Lemma are finite. 
First note that for all $f \in {\cal F}_3$ and  $g$ defined by $g(x):=f(x)-f'(0)x-f''(0)x^2/2$ for $x\in \Rset$ we have for all $\zeta_3$-compatible pairs $(V,W)$ that 
$\E[f(V)-f(W)]=\E[g(V)-g(W)]$. Since $g'(0)=g''(0)=0$ we hence have 
\begin{align*}
\zeta_3(V,W)= \sup_{f\in{\cal F}_3}| \E[f(V)-f(W)]| = \sup_{g\in{\cal F}^*_3}| \E[g(V)-g(W)]|,
\end{align*}
with ${\cal F}^*_3:=\{g\in {\cal F}_3 : g'(0)=g''(0)=0\}$.

For $g\in{\cal F}^*_3$ we have the Taylor expansion 
$g(x+h)=g(x)+g'(x)h+g''(x)h^2/2+R(x,h)$ for all $x,h\in\Rset$ with, using the remainder in integral form, $|R(x,h)|\le |h|^3/6$. Hence, with $V_1,W_1, V_2,W_2$ as in the statement of the Lemma we obtain
\begin{align*}
\zeta_3(V_1+W_1, V_2+W_2)&=  \sup_{g\in{\cal F}^*_3}| \E[g(V_1+W_1)-g(V_2+W_2)]|\\
&=\sup_{g\in{\cal F}^*_3}\left|\E\left[g(V_1)+g'(V_1)W_1+ \frac{g''(V_1)W_1^2}{2}+R(V_1,W_1)   \right.\right.\\
&\left.\left. \phantom{\sup_{g\in{\cal F}^*_3}|\E[}\quad
 - \left(g(V_2)+g'(V_2)W_2+ \frac{g''(V_2)W_2^2}{2}+R(V_2,W_2)\right)\right]\right|\\
 &\le \zeta_3(V_1, V_2) + B,
\end{align*}
with 
\begin{align*}
B&:= \sup_{g\in{\cal F}^*_3}\left|\E\left[g'(V_1)W_1
+ \frac{g''(V_1)W_1^2}{2}+R(V_1,W_1)   \right.\right.\\
 &\left.\left. \phantom{\sup_{g\in{\cal F}^*_3}|\E[}\quad - \left(g'(V_2)W_2+ \frac{g''(V_2)W_2^2}{2}+R(V_2,W_2)\right)\right]\right|\\
&\le  \sup_{g\in{\cal F}^*_3} \sum_{i=1}^2 \left\{ |\E[g'(V_i)W_i]| + \frac{|\E[g''(V_i)W_i^2]|}{2} +\frac{\E[|W_i|^3]}{6} \right\}.
\end{align*}
Since $g'(0)=g''(0)=0$ and $g''$ is Lipschitz-continuous with Lipschitz-constant $1$ we obtain for all $x\in\Rset$ that 
$|g''(x)|=|g''(x)-g''(0)| \le |x|$ and, integrating this inequality, that $|g'(x)|\le  x^2/2$. Hence we obtain 
\begin{align*}
B&\le   \sum_{i=1}^2 \E\left[\frac{V_i^2|W_i|}{2}+\frac{|V_i|W_i^2}{2} +  \frac{|W_i|^3}{6}\right].
\end{align*}
H\"older's inequality implies the assertion.
\end{proof}

\subsection{Two more technical Lemmata}\label{seclem}
The proof of Theorem \ref{main_thm} in Section \ref{pfpf}  requires that $b^{(n)}$ defined in (\ref{def_bn}) tends to $0$ in the $L_3$-norm.   The following Lemma  provides a quantitative estimate.
\begin{lem}\label{bn_con}
For $b^{(n)}$ defined in {\rm (\ref{def_bn})} we have, as $n\to\infty$,
\begin{align}
\| b^{(n)}\|_3 =\bo\left(\frac{1}{\sqrt{\log n}}\right).
\end{align}
\end{lem}
\begin{proof}
We have  
\begin{align*}
\|b^{(n)}\|_3&\le  \frac{1}{\sigma(n)}\left(\left\|\left(\frac{I_n+1}{n+1} - U_1\right) Y^{(0)}\right\|_3+
\left\| \left(\frac{n-I_n}{n+1} -(1-U_1)\right) Y^{(1)}\right\|_3  \right.\\
&~\left. \quad\quad \quad\quad  +\left\|\frac{n}{n+1}C_n(I_n+1) -C(U_1) \right\|_3  \right)\\
&=:  \frac{1}{\sigma(n)}(S_1+S_2+S_3).
\end{align*}
Note that the summands $S_1$ and $S_2$  are equal. Moreover, we have that $(I_n,U_1)$ is independent of $Y^{(0)}$ and $Y^{(1)}$. Hence, we have
\begin{align*}
S_1+S_2= 2 \left\|\frac{I_n+1}{n+1} - U_1 \right\|_3 \|Y\|_3.
\end{align*}
By the Marcinkiewicz--Zygmund inequality, see, e.g.~Chow and Teicher \cite[p.~386]{chte97}, there exists a finite constant $M_3>0$ such that for all $u\in[0,1]$ we have
\begin{align}\label{mazy}
\E\left[|B_{n-1,u}-(n-1)u|^3\right]\le M_3 (n-1)^{3/2}.
\end{align}
Recall that conditionally on $U_1=u$ we have that $I_n$ is binomial $\mathrm{B}(n-1,u)$ distributed. The bound (\ref{mazy}) and integration hence imply  $\|(I_n+1)/(n+1) - U_1 \|_3 = \bo(1/\sqrt{n})$ and $S_1+S_2=\bo(1/\sqrt{n})$.

To bound the summand $S_3$ note that for $S_3= \bo(1/\sqrt{n})$ it is sufficient to show
$\|C_n(I_n+1) -C(U_1) \|_3= \bo(1/\sqrt{n})$. We have
\begin{align*}
\|C_n(I_n+1) -C(U_1) \|_3 \le \left\|C_n(I_n+1) -C\left(\frac{I_n}{n-1}\right) \right\|_3 + \left\|C\left(\frac{I_n}{n-1}\right) -C(U_1) \right\|_3.
\end{align*}
Note that we have  $\|C_n(I_n+1) -C(I_n/(n-1)) \|_3=\bo((\log n)/n)$ using Proposition 3.2 in R\"osler \cite{ro91}. Hence, it remains to bound $\|C(I_n/(n-1)) -C(U_1) \|_3.$ Using symmetry in the terms $x\log x$ and $(1-x)\log(1-x)$ appearing in $C(x)$ and the triangle inequality, we have
\begin{align}\label{est_t1}
\left\|C\left(\frac{I_n}{n-1}\right) -C(U_1) \right\|_3 \le 4\left\|\frac{I_n}{n-1} \log\left(\frac{I_n}{(n-1)U_1} \right)\right\|_3 + 4\left\|\left(\frac{I_n}{n-1}-U_1\right) \log U_1\right\|_3.
\end{align}
To bound the first summand in the latter display we  again use that conditional on $U_1=u$ the random variable $I_n$ has the Binomial B$(n-1,u)$ distribution. Hence
\begin{align}\label{re1}
\left\|\frac{I_n}{n-1} \log\left(\frac{I_n}{(n-1)U_1}\right) \right\|_3^3 = \int_0^1\E\left[ \left|\frac{B_{n-1,u}}{n-1} \log\left(\frac{B_{n-1,u}}{(n-1)u}\right)\right|^3\right]\,du.
\end{align}
To  bound the expectation appearing as integrand in the latter display we consider for $u\in(0,1)$ the  event 
\begin{align*}
E_u:=\left\{ B_{n-1,u}\ge \frac{u}{2}(n-1)\right\}.
\end{align*}
Note that for the complement $E_u^c$ of $E_u$, Chernoff's bound, see \cite{ch52} or \cite[Theorem 1.1]{mcdi98}, yields  $\Prob(E_u^c)\le \exp(-(n-1)u^2/2)$. We denote $h(x):=x\log x$ for $x\in[0,\infty)$. With $\sup_{x\in [0,1/2]}|h(x)| = 1/e\le 1 $ we bound the contribution on $E_u^c$ by
\begin{align}\label{re2}
\int_{E_u^c} \left|\frac{B_{n-1,u}}{n-1} \log\left(\frac{B_{n-1,u}}{(n-1)u}\right)\right|^3\, d\Prob &=
\int_{E_u^c} u^3\left|h\left(\frac{B_{n-1,u}}{(n-1)u}\right)\right|^3\, d\Prob   \nonumber\\
&\le  u^3 \exp\left(-\frac{(n-1)u^2}{2}\right). 
\end{align}
On $E_u$ we apply the mean value theorem to $h(1+y)=h(1+y)-h(1)$ and obtain
\begin{align}\label{re3}
\lefteqn{\int_{E_u} \left|\frac{B_{n-1,u}}{n-1} \log\left(\frac{B_{n-1,u}}{(n-1)u}\right)\right|^3\, d\Prob } \nonumber\\
 &= \int_{E_u} u^3 \left|h\left(1+\frac{B_{n-1,u}-(n-1)u}{(n-1)u}\right)\right|^3\, d\Prob \nonumber \\
&\le \int_{E_u} u^3(1-\log u)^3\left|\frac{B_{n-1,u}-(n-1)u}{(n-1)u}\right|^3\, d\Prob.
\end{align}
With the Marcinkiewicz--Zygmund inequality (\ref{mazy}) we can further estimate the integral in (\ref{re3}) and obtain
\begin{align}\label{re3b}
\int_{E_u} \left|\frac{B_{n-1,u}}{n-1} \log\left(\frac{B_{n-1,u}}{(n-1)u}\right)\right|^3\, d\Prob 
 \le M_3\frac{(1-\log u)^3}{(n-1)^{3/2}}.
\end{align}
Hence, plugging (\ref{re2}) and (\ref{re3b}) into (\ref{re1}) we have
\begin{align}
\left\|\frac{I_n}{n-1} \log\left(\frac{I_n}{(n-1)U_1}\right) \right\|_3^3 
&\le \int_0^1 \left\{ u^3 \exp\left(-\frac{(n-1)u^2}{2}\right) +  M_3 \frac{(1-\log u)^3}{(n-1)^{3/2}} \right\} \,du \nonumber\\
&= \bo\left(\frac{1}{n^2}\right) + \bo\left(\frac{1}{n^{3/2}}\right)  \label{okamoto}\\
&= \bo\left(\frac{1}{n^{3/2}}\right). \nonumber 
\end{align}
The second summand in (\ref{est_t1}) is also estimated by use of the bound  (\ref{mazy}):
\begin{align*}
\left\|\left(\frac{I_n}{n-1}-U_1\right) \log U_1\right\|_3^3
&= \int_0^1 \E\left[\left| \frac{B_{n-1,u}}{n-1}-u\right|^3\right] |\log u|^3 \,du\\
&\le \int_0^1 M_3 \frac{|\log u|^3}{(n-1)^{3/2}}  \,du\\
&=\bo \left(\frac{1}{n^{3/2}}\right).
\end{align*}
Altogether, we have $S_3=\bo(1/\sqrt{n})$, hence
 $S_1+ S_2+S_3=\bo(1/\sqrt{n})$. Since  $\sigma(n)=\Omega(\sqrt{\log n}/\sqrt{n})$  the assertion follows.
\end{proof}

Moreover the proof of Theorem \ref{main_thm}  in Section \ref{pfpf}  requires an initial estimate for the $L_3$-norm  $\|Y_n-Y\|_3$. Note that the following Lemma \ref{lem_l3} is improved later by proving Corollary \ref{coro}.
\begin{lem} \label{lem_l3}
For the error term $Y_n-Y$ in Theorem \ref{main_thm} we have, as $n\to\infty$,
\begin{align} \label{err_l3}
\|Y_n-Y\|_3 = \bo\left(\sqrt{\frac{\log n}{n}}\right).
\end{align}
\end{lem}
\begin{proof}
Since $Y_n$ is a bounded random variable and $Y$ has finite absolute moments of arbitrary order we have $\|Y_n-Y\|_3<\infty$ for all $n\ge 0$.
Note that with $X_n$ defined in (\ref{scaling}) the assertion (\ref{err_l3}) is equivalent to $\E[|X_n|^3] =\bo(1)$. From (\ref{basic_eqn}) we obtain
\begin{align*}
|X_n|\le \Lambda_0 + \Lambda_1 + |b^{(n)}|
\end{align*}
with 
\begin{align*}
\Lambda_0:=A_0^{(n)} \frac{1}{\sigma(I_n)}\left|Y_{n,0}-Y^{(0)}\right|, \quad 
\Lambda_1:=A_1^{(n)}\frac{1}{\sigma(n-1-I_n)}\left|Y_{n,1}-Y^{(1)}\right|.
\end{align*}
Hence, we have for all $n\ge 1$ that 
\begin{align}
\E\left[|X_n|^3\right]&\le \E\left[\Lambda_0^3\right]+ \E\left[\Lambda_1^3\right] + \E\left[|b^{(n)}|^3\right]  + 3\E\left[\Lambda_0^2\Lambda_1\right]+
3\E\left[\Lambda_0\Lambda_1^2\right] \nonumber \\
&\quad ~ + 3\E\left[\Lambda_0^2|b^{(n)}|\right] + 3\E\left[\Lambda_0|b^{(n)}|^2\right]+  
3\E\left[\Lambda_1^2|b^{(n)}|\right] + 3\E\left[\Lambda_1|b^{(n)}|^2\right]   \label{mittel}\\
&\quad ~ + 6\E\left[\Lambda_0\Lambda_1|b^{(n)}|\right]. \nonumber
\end{align}
We use the notation 
\begin{align*}
\beta_n:= 1\vee \max_{0\le j\le n} \E\left[|X_j|^3\right].
\end{align*}
We start bounding the previous sum with the summand $\E[\Lambda_0^3]$. For all $0\le j\le n-1$, conditionally given $I_n=j$ we have that 
 $A_0^{(n)}$ is deterministic and $|Y_{n,0}-Y^{(0)}|/\sigma(I_n)$ is distributed as $|X_j|$. Hence we obtain 
\begin{align} \label{aaaa}
\E\left[\Lambda_0^3\right] \le \E\left[\left(A_0^{(n)}\right)^3\right] \beta_{n-1}
\end{align}
and an analogous bound for $\E[\Lambda_1^3]$. The summand $\E[|b^{(n)}|^3]$ tends to zero by Lemma \ref{bn_con}. For the summand $\E[\Lambda_0^2\Lambda_1]$ first note that again by conditioning on  $I_n=j$ we have independence of $|Y_{n,0}-Y^{(0)}|/\sigma(I_n)$ and $|Y_{n,1}-Y^{(1)}|/\sigma(n-1-I_n)$ with distributions of $|X_j|$ and $|X_{n-1-j}|$,  respectively. Since moreover $A_0^{(n)}$ and $A_1^{(n)}$ are uniformly bounded we obtain for an appropriate constant $0<D<\infty$ that
\begin{align*}
\E[\Lambda_0^2\Lambda_1] \le D \left(\max_{0\le j\le n-1} \|X_j\|_2^2 \right) \left(\max_{0\le j\le n-1} \|X_j\|_1 \right).
\end{align*}
Note that (\ref{var_asy_exp}) implies $\sup_{n\ge 0} \|X_n\|_2<\infty$, hence we have $\E[\Lambda_0^2\Lambda_1]=\bo(1)$. Analogously, $\E[\Lambda_0\Lambda_1^2]=\bo(1)$.
The summands in line (\ref{mittel}) are all bounded by H\"older's inequality, e.g., for  the first of these summands we have, also using (\ref{aaaa}), Lemma \ref{bn_con} and (\ref{asy_norm}), that for all $n$ sufficiently large
\begin{align*}
\E\left[\Lambda_0^2|b^{(n)}|\right]\le \|\Lambda_0\|_3^2 \|b^{(n)}\|_3 \le \beta_{n-1}^{2/3} \|b^{(n)}\|_3 \le  \beta_{n-1} \|b^{(n)}\|_3 = o(1)\beta_{n-1}.
\end{align*}
The other summands in line  (\ref{mittel}) yield the same contribution. Finally, we similarly have 
\begin{align*}
\E[\Lambda_0\Lambda_1|b^{(n)}|] \le \|\Lambda_0\|_3 \|\Lambda_1\|_3\|b^{(n)}\|_3= o(1)\beta_{n-1}.
\end{align*}
Collecting all terms we obtain 
\begin{align} \label{basic_ee}
\E\left[|X_n|^3\right]&\le \left(\E\left[\left(A_0^{(n)}\right)^3 + \left(A_1^{(n)}\right)^3\right] +o(1)\right) \beta_{n-1} + \bo(1).
\end{align}
With the asymptotic result (\ref{asy_norm}) this implies
\begin{align*}
\E\left[|X_n|^3\right]&\le \left(\E\left[U_1^{3/2} + (1-U_1)^{3/2}\right] +o(1)\right) \beta_{n-1} + \bo(1) = \left(\frac{4}{5} +o(1)\right) \beta_{n-1} + \bo(1).
\end{align*}
Hence, there exist an $n_0\in \Nset$ and a constant $0<D'<\infty$ such that for all $n\ge n_0$ we have 
\begin{align*}
\E\left[|X_n|^3\right]\le \frac{9}{10}\beta_{n-1} + D'.
\end{align*}
It is easy to check by induction that $\E\left[|X_n|^3\right]\le \beta_{n_0}\vee (10D')$ for all $n\ge 0$, hence
 $\E[|X_n|^3]=\bo(1)$, as $n\to\infty$.
\end{proof}

\noindent
{\bf Remark.} The argument of the proof of Lemma \ref{lem_l3} can be extended by induction  on $p$ to show, as $n\to\infty$,
\begin{align*} 
\|Y_n-Y\|_p = \bo\left(\sqrt{\frac{\log n}{n}}\right)
\end{align*}
for any $1\le p <\infty$. A related induction argument for a bound of the minimal $L_p$-metric $\ell_p(Y_n,Y)$ is given in Fill and Janson \cite[Section 3]{FJ02}.

\subsection{The proof of Theorem \ref{main_thm}}\label{pfpf}
We now prove Theorem \ref{main_thm} and Corollary \ref{coro}.
\begin{proof}[Proof of Theorem \ref{main_thm}]
We first define a ``hybrid" random variable to connect between $X_n$ and a standard normal random variable as follows: For ${\cal N}^{(0)}$ and ${\cal N}^{(1)}$ independent standard normal random variables also independent of all other random variables, i.e., independent of $(U_i)_{i\ge 1}$,
we set
\begin{align*}
Q_n:= A_0^{(n)} {\cal N}^{(0)}
+ A_1^{(n)} {\cal N}^{(1)}, \quad n\ge 1.
\end{align*}
Note that (\ref{asy_norm}) with $p=2$ implies that $\V(Q_n)\to 1$ as $n\to \infty$. Further, we have $\V(Q_n)>0$ for all $n\ge 1$.
Hence, there exists a (deterministic) sequence $(\kappa_n)_{n\ge 1}$ with $\kappa_n\to 0$ 
as $n\to \infty$ such that $\V((1+\kappa_n)Q_n)=1$ for all $n\ge 1$. Denoting 
by ${\cal N}$ another standard normal random variable we have that each 
pair from the three random variables $X_n$, $(1+\kappa_n)Q_n$ and ${\cal N}$ is $\zeta_3$-compatible. Thus, we can use the triangle inequality to obtain 
\begin{align}\label{tria}
\zeta_3(X_n,{\cal N}) \le \zeta_3(X_n,(1+\kappa_n)Q_n) + \zeta_3((1+\kappa_n)Q_n, {\cal N}).
\end{align}
For $n\ge 1$ we now introduce the abbreviations
\begin{align*}
Z_n^{(0)}:=\frac{1}{\sigma(I_n)}(Y_{n,0}-Y^{(0)}), \quad   Z_n^{(1)}:=\frac{1}{\sigma(n-1-I_n)}(Y_{n,1}-Y^{(1)})
\end{align*}
and
\begin{align*}
\Phi_n:=A^{(n)}_0 Z^{(0)}_n+A^{(n)}_1 Z^{(1)}_n.
\end{align*}
Then, Lemma  \ref{zolo_est} can be applied to the sums
\begin{align*}
X_n=\Phi_n+ b^{(n)},\quad  (1+\kappa_n)Q_n= Q_n+\kappa_n Q_n 
\end{align*}
and yields
\begin{align*}
 \zeta_3(X_n,(1+\kappa_n)Q_n) \le \zeta_3(\Phi_n, Q_n)&+\frac{1}{2}\|\Phi_n\|_3^2 \|b^{(n)}\|_3+\frac{1}{2} \|\Phi_n\|_3 \|b^{(n)}\|_3^2 + \frac{1}{6}\|b^{(n)}\|_3^3\\
&~+ \left(\frac{1}{2}|\kappa_n| + \frac{1}{2}\kappa_n^2 + \frac{1}{6} |\kappa_n|^3\right)\|Q_n\|_3^3.   
\end{align*}
Note that by definition of $Q_n$ we have $\sup_{n\ge 1} \|Q_n\|_3 <\infty$. Moreover, Lemma \ref{lem_l3} implies that  $\sup_{n\ge 1} \|\Phi_n\|_3 <\infty$.  
Hence, with $\kappa_n \to 0$ and, by Lemma \ref{bn_con}, $\|b^{(n)}\|_3 \to 0$ 
we obtain, as $n\to\infty$, 
\begin{align}\label{tria2}
 \zeta_3(X_n,(1+\kappa_n)Q_n) \le \zeta_3(\Phi_n, Q_n)+ o(1).
\end{align}
Next we show that for the second summand in (\ref{tria}) we  have $\zeta_3((1+\kappa_n)Q_n, {\cal N})=o(1)$: First note that $\sup_{n\ge 1} \|Q_n\|_3 <\infty$  implies that the $L_3$-norm of $(1+\kappa_n)Q_n$ is uniformly bounded in $n$. Hence, the bound (\ref{est_zeta_ell}) implies $\zeta_3((1+\kappa_n)Q_n, {\cal N})\le M \ell_3((1+\kappa_n)Q_n, {\cal N})$ for all $n\ge 0$ and a fininte constant $M>0$. Using the uniform $U_1$ in (\ref{asy_norm}) (that is also independent of ${\cal N}^{(0)}$ and ${\cal N}^{(1)}$) we have that  $\sqrt{U_1} {\cal N}^{(0)} + \sqrt{1-U_1} {\cal N}^{(1)}$ has also the standard normal distribution. Hence we obtain
\begin{align}
\zeta_3((1+\kappa_n)Q_n, {\cal N})&\le M \ell_3((1+\kappa_n)Q_n, {\cal N})\nonumber\\
&\le M\left\|  \left((1+\kappa_n) A_0^{(n)} -\sqrt{U_1}\right){\cal N}^{(0)} +  \left((1+\kappa_n) A_1^{(n)} -\sqrt{1-U_1}\right){\cal N}^{(1)} \right\|_3 \nonumber\\
&\to 0, \label{rrr}
\end{align}
by independence and (\ref{asy_norm}).

Hence,  we obtain
from (\ref{tria}), (\ref{tria2})  and (\ref{rrr}) that
\begin{align}\label{cone0}
\zeta_3(X_n,{\cal N}) \le \zeta_3(A^{(n)}_0 Z^{(0)}_n+A^{(n)}_1 Z^{(1)}_n, A^{(n)}_0 {\cal N}^{(0)}+A^{(n)}_1{\cal N}^{(1)}) +o(1).
\end{align}
Now, note that for all $0\le k\le n-1$, conditionally given $I_n=k$ we have that 
 $Z_n^{(0)}$ and $Z_n^{(1)}$ are independent with distributions of $X_k$ and $X_{n-1-k}$, respectively. By $(X^{(0)}_0,\ldots,X^{(0)}_{n-1})$, $(X^{(1)}_0,\ldots,X^{(1)}_{n-1})$ independent vectors with identical distribution $(X_0,\ldots,X_{n-1})$ are denoted. Thus, conditioning on $I_n$ and using that $\zeta_3$ is $(3,+)$-ideal and (\ref{sum_bound}), we obtain 
\begin{align} \label{cone1}
\lefteqn{\zeta_3(A^{(n)}_0 Z^{(0)}_n+A^{(n)}_1 Z^{(1)}_n, A^{(n)}_0 {\cal N}^{(0)}+A^{(n)}_1{\cal N}^{(1)}) } \nonumber \\
&\le \frac{1}{n}\sum_{k=0}^{n-1} \zeta_3\left(\frac{(k+1)\sigma(k)}{(n+1)\sigma(n)}X_k^{(0)} +\frac{(n-k)\sigma(n-1-k)}{(n+1)\sigma(n)} X_{n-1-k}^{(1)}, \right. \nonumber\\
&\left. \phantom{\frac{1}{n}\sum_{k=0}^{n-1} \zeta_3} \quad\quad
 \frac{(k+1)\sigma(k)}{(n+1)\sigma(n)} {\cal N}^{(0)}+\frac{(n-k)\sigma(n-1-k)}{(n+1)\sigma(n)}{\cal N}^{(1)}\right) \nonumber \\
 &\le  \frac{1}{n}\sum_{k=0}^{n-1}\left\{ \left(\frac{(k+1)\sigma(k)}{(n+1)\sigma(n)}\right)^3 \zeta_3(X_k,{\cal N}) + \left(\frac{(n-k)\sigma(n-1-k)}{(n+1)\sigma(n)}\right)^3\zeta_3(X_{n-1-k},{\cal N})\right\} \nonumber \\
 &=  \frac{1}{n}\sum_{k=0}^{n-1} 2 \left(\frac{(k+1)\sigma(k)}{(n+1)\sigma(n)}\right)^3 \zeta_3(X_k,{\cal N}).
\end{align}
With $\Delta(n):= \zeta_3(X_n,{\cal N})$ we obtain from (\ref{cone0}) and (\ref{cone1}) that 
\begin{align} \label{dist_e}
\Delta(n)\le \E\left[ 2\left(\frac{(I_n+1)\sigma(I_n)}{(n+1)\sigma(n)}\right)^3\Delta(I_n)\right]+o(1).
\end{align}
Now, a standard argument implies $\Delta(n)\to 0$ as follows: Note that $\sigma(n)\sim \sqrt{2\log(n)/n}$ and that $I_n$ is distributed uniformly on $\{0,\ldots,n-1\}$ imply for $U$ uniformly distributed on $[0,1]$ that
\begin{align} \label{cont_fac}
\E\left[ 2\left(\frac{(I_n+1)\sigma(I_n)}{(n+1)\sigma(n)}\right)^3\right] \to \E\left[2 U^{3/2}\right]=\frac{4}{5}<1.
\end{align}
First we use (\ref{dist_e}) for a rough  bound:
\begin{align*} 
\Delta(n)\le \E\left[ 2\left(\frac{(I_n+1)\sigma(I_n)}{(n+1)\sigma(n)}\right)^3\right]\sup_{0\le k\le n-1} \Delta(k)+o(1).
\end{align*}
In view of (\ref{cont_fac}) this implies, similarly to the last four lines of the proof of Lemma \ref{lem_l3}, that $(\Delta(n))_{n\ge 0}$ is bounded. We denote $\eta:=\sup_{n\ge 0} \Delta(n)<\infty$ and $\lambda :=\limsup_{n\to\infty} \Delta(n)\ge 0$. For any $\varepsilon>0$ there exists an $n_0 \ge 0$ such that $\Delta(n)\le \lambda + \varepsilon$ for all $n\ge n_0$. Hence, from (\ref{dist_e}) we obtain
\begin{align*}
\Delta(n)&\le \E\left[ {\bf 1}_{\{I_n<n_0\}}2\left(\frac{(I_n+1)\sigma(I_n)}{(n+1)\sigma(n)}\right)^3\right]\eta \\
&\quad ~+ \E\left[ {\bf 1}_{\{I_n\ge n_0\}}2\left(\frac{(I_n+1)\sigma(I_n)}{(n+1)\sigma(n)}\right)^3\right] (\lambda+\varepsilon) +  o(1).
\end{align*}
With $n\to \infty$ this implies
\begin{align}
\lambda = \limsup_{n\to\infty}\Delta(n)\le \frac{4}{5} (\lambda+\varepsilon).
\end{align}
Since $\varepsilon>0$ is arbitrary this implies $\lambda=0$. Hence, we have $\zeta_3(X_n,{\cal N})\to 0$ as $n\to \infty$. Since convergence in $\zeta_3$ implies weak convergence, the assertion follows.
\end{proof}

\begin{proof}[Proof of Corollary \ref{coro}]
Note that in the proof of Theorem \ref{main_thm} with $X_n=(Y_n-Y)/\sigma(n)$ the 
convergence $\zeta_3(X_n,{\cal N})\to 0$ is shown. This implies $\E[|X_n|^3] \to \E[|{\cal N}|^3]$ as $n\to \infty$, since the function $x\mapsto |x|^3/6$ is an element of ${\cal F}_3$. 
Hence we obtain 
\begin{align*}
\left\|Y_n - Y\right\|_3 =   \sigma(n)\|X_n\|_3   \sim  \sqrt{\frac{2\log n}{ n}}\|{\cal N}\|_3=
  \frac{2}{\pi^{1/6}} \sqrt{\frac{\log n}{ n}},
\end{align*}
the assertion.
\end{proof}

\end{document}